\documentclass{amsart}
\usepackage{graphicx} 
\usepackage{enumerate,amsthm,amsmath,amsfonts,amssymb,mathrsfs,hyperref,dutchcal,xcolor}
\usepackage[cal=cm]{mathalfa}
\usepackage{fullpage}
\newtheorem{thm}{Theorem}
\newtheorem{defi}{Definition}

\newtheorem{rmk}{Remark}

\title{Phonon gap analysis for equilibria of perturbed almost-periodic Frenkel-Kontorova models}
\author{Yujia An}
\address{School of Mathematical Sciences, Beijing Normal University,
No. 19, XinJieKouWai St., HaiDian District, Beijing 100875, P. R. China}
\email{yjan@mail.bnu.edu.cn}

\author{Xifeng Su}
\address{School of Mathematical Sciences, Laboratory of Mathematics and Complex Systems (Ministry of Education)\\
Beijing Normal University,
No. 19, XinJieKouWai St., HaiDian District, Beijing 100875, P. R. China}
\email{xfsu@bnu.edu.cn, billy3492@gmail.com}

\date{\today}

\begin{document}
\maketitle

\begin{abstract}
For generalized Frenkel-Kontorova models  subjected to almost-periodic media, employing both the KAM method and the approach of `anti-integrable' limits, two different types of equilibria
are obtained in \cite{an2024kamtheoryalmostperiodicequilibria} and \cite{du2024anti} respectively. 

We study the phonon gap around these equilibria  and we find that the KAM equilibria do not have a phonon gap but the equilibria obtained by anti-integrable limits do have.
\end{abstract}

\section{Introduction}
The celebrated Kolmogorov–Arnold–Moser (KAM for short) theory, which focuses on perturbations of integrable systems, and the other perturbation theory for fully chaotic systems far away from the integrable,
which is called `anti-integrable' limits, provide two different mechanisms to search for the equilibria of the standard Frenkel-Kontorova (FK for short) models on periodic media. See e.g. \cite{Rafaelbook} and references therein, and \cite{MR1067910, VeermanT91, Aubry91} for an introduction of KAM theory and anti-integrable limits respectively in the present context.

We point out that these equilibria with periodic potentials can be interpreted as orbits of some twist map on the annulus and so it is natural to study the hyperbolicity of them in a dynamical way. For instance, \cite{Goroff85} provides hyperbolic sets\footnote{Actually, these sets contain the equilibria derived from the idea of anti-integrable limits.} of twist maps corresponding to some specific examples and \cite{Mackay_1992} shows these sets can be hyperbolic cantori, which could be thought
of as remnants of KAM tori. 

As for the non-perturbative approaches, it is worthwhile to mention the celebrated Aubry–Mather theory developed independently by Aubry \cite{aubry1983} and Mather \cite{Mather82} in the early 1980s, which utilizes
variational methods to look for ground states. In 1988,  Le Calvez proved in \cite{LeCalvez88}, that for generic exact symplectic twist map, there exists an open dense subset of rotation numbers such that the corresponding Aubry-Mather set is hyperbolic. Later in \cite{Arnaud11}, Arnaud gives an equivalence condition for which the Aubry-Mather set is uniformly hyperbolic by the Green bundles tools.\footnote{See \cite{Arnaud16} for more discussions for Lyapunov exponents and Green bundles.}
The well understanding of the hyperbolicity of the Aubry-Mather may be helpful for some problems in \cite{Herman98}.

From another perspective in \cite{AUBRY1992123}, Aubry, Mackay, and Baesens observe that the notion of the uniform hyperbolicity for symplectic twist maps is closely related to that of the phonon gap.  More precisely, they show  that the set of orbits of the symplectic twist map being uniformly hyperbolic is equivalent to the associated sets of equilibrium states having a phonon gap. Note that  the phonon gap is an important concept in solid physics: 
an equilibrium configuration that exhibits a phonon gap implies the presence of an energy barrier, preventing it from transitioning to any other equilibrium state, thereby signifying a stable structure.

We remark that in the interpretation of deposition, the existence of quasi-periodic KAM solutions implies the existence of a continuum of equilibria (see for instance \cite{SuL17}), so that the system can slide, which may seem to be caused by zero phonon gap in the periodic case. (In contrast, if the KAM tori are not present, the system is pinned.) However, in the quasi-periodic FK models, \cite{DSZ16} shows the resonant quasi-periodic equilibria can be pinned even if the phonon gap is zero. Therefore, we think the phonon gap analysis would be an independent object to study with.

The main purpose of this note is to investigate whether the equilibrium configurations obtained through two different methods in the quasi-periodic or almost-periodic FK model possess a phonon gap. 
See Theorems~\ref{KAM equilibria} and \ref{anti-integrable limits}.
We show that the phonon gap vanishes in the equilibrium configurations obtained by the KAM method, while the equilibrium configurations obtained by the anti-integrable limit have a phonon gap. 

Let us recall that in the standard FK model, the configuration of the system is given by a sequence $\{u_n\}_{n\in\mathbb Z}$, where $u_n\in\mathbb R^d$ for each $n\in \mathbb{Z}$, and the formal energy of the sequence $\{u_n\}_{n\in\mathbb Z}$ is given by 
$$\Phi(u)=\sum_{n\in\mathbb Z} \left[ \frac12(u_n-u_{n+1})^2+V(u_n)\right].$$
We are interested in the equilibrium configurations where the derivatives of the formal energy with respect to the positions $u_n$ vanish.

There are several physical interpretations of the FK models. The original motivation \cite{FK39} (see also \cite{BK2004})  was dislocations in solids. 
In the interpretation of \cite{aubry1983}, $u_n$ denote the positions of a deposited material over a substratum. 
The interaction of the atoms with the substratum is modeled by the term $V$. The periodicity of $V$ considered in \cite{aubry1983} corresponds to a periodic substratum (e.g., a crystal).
 Such periodic models rely on the description of the orbits of some twist map by means of a generating function.
 

However, when $V$ is a quasi-periodic or an almost-periodic function\footnote{The associated models are called generalized FK models in the present paper.}, 
the corresponding system does not admit an easy dynamical formulation. 
Therefore, the phonon gap would be a good substitute candidate for hyperbolicity to analyze the behavior of the equilibria.
We remark that the quasi-periodic models considered in \cite{MR3023435} could appear in cleaved faces of crystals or in quasi crystals, and the phonon localization of the KAM equilibria and the equilibria obtained from anti-integrable limits is studied in \cite{DSZ16} and \cite{du2024anti} respectively.
For numerical explorations of these issues, one may refer e.g. \cite{vanErp'99a, vanErp'01}.

We point out that  extending the Aubry–Mather theory to quasi-periodic or almost-periodic media presents significant challenges, and we will not go in this direction of the non-perturbative problems in the present paper.
However, it would be  an interesting question to ask: whether there is an equivalence condition of phonon gap in the aperiodic and non-perturbative setting.

In this note, we mainly study the equilibria obtained by the KAM method \cite{an2024kamtheoryalmostperiodicequilibria} and  the anti-integrable limits \cite{du2024anti}. 
Both equilibria correspond to different potential $V$ subject to almost-periodic media. Indeed, we will show that
\begin{itemize}
\item [(i)]  the KAM method, for instance the potential $V$ is sufficiently small and the rotation number is Diophantine, there exists an equilibrium without a phonon gap. 
\item [(ii)] under the assumptions of the anti-integrable limits, for example the potential $V$ satisfies the Aubry criterion and is large enough\footnote{This equilibrium can indeed be found in any dimension.},
there exists an equilibrium with a phonon gap.
\end{itemize}

\begin{rmk}
We point out that the set of functions satisfying the Aubry criterion in \cite{du2024anti} is a rather big set in the sense that it contains periodic, quasi-periodic, almost-periodic and  pattern-equivariant functions\footnote{Some simple example could be found in e.g. \cite{GGP06}.}.
Hence, we actually provide a concise proof for the existence of equilibria with a phonon gap for rather general models of FK type in the context of the anti-integrable limits.

Note also that although we only prove the KAM equilibria without a phonon gap for the almost-periodic FK models similar arguments would work for periodic and quasi-periodic models. However, KAM approach for the FK models with pattern-equivariant potential is still unknown.
\end{rmk}

\section{The phonon gap}
In the generalized FK model, the energy of configuration $u=\{u_n\}_{n\in\mathbb Z}$, $u_n\in\mathbb R^d$ is given by $$\Phi(u)=\sum_{n}L(u_n,u_{n+1}),$$ where the Lagrangian $L:\mathbb R^d\times\mathbb R^d\to \mathbb R$ is a $C^2$ function, which is slightly more general than the one in the introduction. We collect several definitions below related to the notion of phonon gap, which are prepared for our later analysis.

\begin{defi}
    The configuration $u=\{u_n\}_{n\in\mathbb Z}$, $u_n\in\mathbb R^d$ is called equilibrium states if $D\Phi(u)=0$, that is, $D_1L(u_n,u_{n+1})+D_2L(u_{n-1},u_n)=0$, for all $n\in\mathbb Z$. 
\end{defi}

\begin{defi}
    Given a bounded linear operater $A:\ell^2(\mathbb Z)\to\ell^2(\mathbb Z)$, let $\sigma(A)$ denote its spectrum. We define the spectral maximum as $\sigma_{\max}(A):=\{|E|:E\in\sigma(A)\}$, and spectral minimum as $\sigma_{\min}(A):=\inf\{|E|:E\in\sigma(A)\}$. 
\end{defi}

\begin{defi}
    The phonon spectrum of an equilibrium state $u$ is $\sigma(D^2\Phi(u))$, the spectrum of the Hessian operator $D^2\Phi(u):\ell^2(\mathbb{Z}) \to \ell^2(\mathbb{Z})$,
    \begin{equation*}
    (D^2\Phi(u)\xi)_n = D_{21}L(u_{n-1},u_n)\xi_{n-1} + (D_{11}L(u_n,u_{n+1}) + D_{22}L(u_{n-1},u_n))\xi_n + D_{12}L(u_n,u_{n+1})\xi_{n+1},
    \end{equation*}
    for any $\xi = (\xi_n)_{n\in\mathbb{Z}} \in \ell^2(\mathbb{Z})$. 
\end{defi}

Notice that the operator $D^2\Phi(u)$ is a one-dimensional Schr\"odinger operator with a position dependent potential. 
The dependence will be given by the dynamics of the $u_n$. 
In particular, for the solutions given by a hull function, we will be considering quasi-periodic or even almost-periodic potentials.
The mathematical theory of the spectrum of quasi-periodic or almost-periodic Schr\"odinger operators is well
developed \cite{HaroL, PasturF92}. 

\begin{defi}
    The equilibrium state $u$ has a phonon gap if $0\notin\sigma(D^2\Phi(u))$, that is, $D^2\Phi(u):\ell^2(\mathbb{Z}) \to \ell^2(\mathbb{Z})$ is invertible.   
\end{defi}

\begin{defi}
    The gap parameter of an equilibrium state $u$ is the ratio $G(u):=\sigma_{\min}(D^2\Phi(u))/\sigma_{\max}(D^2\Phi(u))$.
\end{defi}

\begin{rmk}
    If the equilibrium state $u$ does not have a phonon gap, then $G(u)=0$. And if the equilibrium state $u$ has a phonon gap, then $\sigma_{\min}(D^2\Phi(u))=\sigma_{\max}(D^2\Phi(u)^{-1})^{-1}$ and $G(u)=\frac{1}{\sigma_{\max}(D^2\Phi(u)^{-1})\sigma_{\max}(D^2\Phi(u))}$.
\end{rmk}





\section{KAM equilibria without a phonon gap}
In this section we aim to prove that the equilibrium configuration obtained using the KAM method does not have a phonon gap.

In the one dimensional FK model, when the Lagrangian $L(x,y)$ is given by $L(x,y)=\frac{1}{2}(x-y)^2+V(x)$, $x,y\in\mathbb R$, the equilibrium configuration satisfies 
\begin{equation}\label{1}
    u_{n+1}+u_{n-1}-2u_n+V'(u_n)=0.
\end{equation}
When the potential $V$ is almost periodic function $V(\theta)=\widehat V(\theta\alpha)$, $\widehat V:\mathbb T^{\mathbb N}\to\mathbb R$, $\alpha\in\mathbb R^\mathbb N$ is rationally independent, given an $\omega\in\mathbb R_+$, one can find an equilibrium configuration $u_n=h^*(n\omega)$, where $h^*(\theta)=\theta+\tilde{h}^*(\theta)$, with $\tilde{h}^*(\theta)=\hat h^*(\theta\alpha)$, and $\hat h^*:\mathbb T^\mathbb N\to \mathbb R$.
See \cite{an2024kamtheoryalmostperiodicequilibria}.

Actually, the solution $\hat h^*$ satisfies 
\begin{equation*}
    \mathcal{E}[\hat h^*]=\hat h^*(\sigma+\omega\alpha)+\hat h^*(\sigma-\omega\alpha)-2\hat h^*(\sigma)+\partial_{\alpha}\widehat V(\sigma+\alpha\cdot\hat h^*(\sigma))=0,\quad \forall\sigma\in\mathbb T^\mathbb N.
\end{equation*}
Furthermore, for all $\beta\in\mathbb R$, define $\hat h^*_{\beta}(\sigma)=\hat h^*(\sigma+\beta\alpha)+\beta$, we have
\begin{equation}
    \mathcal{E}[\hat h^*_{\beta}]=\hat h^*(\sigma+\omega\alpha+\beta\alpha)+\beta+\hat h^*(\sigma-\omega\alpha+\beta\alpha)+\beta-2\hat h^*(\sigma+\beta\alpha)-2\beta+\partial_{\alpha}\widehat V(\sigma+\alpha\hat h^*(\sigma+\beta\alpha)+\alpha\beta)=0.
\end{equation}
Therefore, we get a family of solutions $u_n(\beta)=h^*_{\beta}(n\omega)=n\omega+\hat h^*(n\omega\alpha+\beta\alpha)+\beta$ that continuously depend on the parameter $\beta$. 

    In physical terms, this indicates that the obtained solutions can slide. That is, applying an arbitrarily small tilting force to all particles can transition them from one equilibrium state to another. However, an equilibrium state has a phonon gap means that there exists an energy barrier to move from that equilibrium state to another. Therefore, it can be inferred that the equilibrium configurations obtained using the KAM method do not possess a phonon gap. We will provide a rigorous mathematical proof. First, we present a simplified version of KAM theorem. The notation is the same as in \cite{an2024kamtheoryalmostperiodicequilibria}.

\begin{thm}[\cite{an2024kamtheoryalmostperiodicequilibria}]\label{sr}
Let $h(\theta)=\theta+\tilde h(\theta)$, $\tilde h(\theta)=\hat h(\alpha\theta)=\sum_{k\in\mathbb Z^\mathbb{N}_*}\hat h_ke^{ik\cdot\alpha\theta}$, $\hat h_0=0$, $\hat h\in\mathscr A_{\rho_0}^1$. $\alpha\in[0,1]^{\mathbb N}$ is rationally independent. $\partial_{\alpha}\widehat V\in\mathscr A_{\rho_0+\|\hat h\|_{\rho_0}+\iota}^2$, $\iota>0$. Denote $\hat l=1+\partial_\alpha\hat h$, $T_{x}(\sigma)=\sigma+x$. We assume the following: 
\begin{itemize}
\item[{\rm (H1)}]
Diophantine condition: $|\omega\alpha\cdot k-2n\pi|\ge\frac\nu{\prod_{j\in\mathbb N}(1+\langle\langle j\rangle\rangle^{1+\tau}|k_j|^{1+\tau})}$ holds for $\forall k\in\mathbb Z^\mathbb{N}_*\setminus\{0\}$, $\forall n\in\mathbb Z$, where $\tau,\nu>0$.
\item[{\rm (H2)}]
Non-degeneracy condition: $\|\hat l(\sigma)\|_{\rho_0}\le N^+$, $\|(\hat l(\sigma))^{-1}\|_{\rho_0}\le N^-$, $\big|\big<\frac1{\hat l\cdot\hat l\circ T_{-\omega\alpha}}\big>\big|\ge c>0$.
\end{itemize}
If $\|\mathcal E[\hat h]\|_{\rho_0}$ is small enough $(\|\mathcal E[\hat h]\|_{\rho_0}\le\epsilon\le\epsilon^*(N^+,N^-,c,\tau,\nu,\iota,\rho_0,\|\partial_{\alpha}\widehat V\|_{\mathscr A_{\rho_0+\|\hat h\|_{\rho_0}+\iota}^2})$, then there exists an analytic function $\hat h^*\in\mathscr A_{\frac{\rho_0}{2}}$, such that 
$$\mathcal E[\hat h^*]=0.$$
Moreover, $$\|\hat h-\hat h^*\|_\frac{\rho_0}{2}\le C_1(N^+,N^-,c,\tau,\nu,\iota,\rho_0,\|\partial_{\alpha}\widehat V\|_{\mathscr A_{\rho_0+\|\hat h\|_{\rho_0}+\iota}^2})\epsilon_0.$$
The solution $\hat h^*$ is the only solution of $\mathcal{E}[\hat h^*]=0$ with zero average for $\hat h^*$ in a ball centered at $\hat h$ in $\mathscr A_{\frac{3}{8}\rho_0}$, i.e. $\hat h^*$ is the unique solution in the set
$$\left\{ \hat g\in\mathscr A_{\frac{3}{8}\rho_0}: \langle \hat g\rangle=0, \|\hat g-\hat h\|_{\frac{3}{8}\rho_0}\le C_3(N^+,N^-,c,\tau,\nu,\iota,\rho_0,\|\partial_{\alpha}\widehat V\|_{\mathscr A_{\rho_0+\|\hat h\|_{\rho_0}+\iota}^2})  \right\}.$$
\end{thm}

\begin{rmk}
    Since $\partial_{\alpha}\widehat V$ is derivative, by vanishing lemma, the extra parameter in   \cite[Theorem 1]{an2024kamtheoryalmostperiodicequilibria} vanishes.
\end{rmk}

\begin{thm}\label{KAM equilibria}
    The equilibrium configuration obtained by Theorem \ref{sr} does not have a phonon gap.
\end{thm}

\begin{proof}
    The direct calculation yields
    $$(D^2\Phi(u)(\xi))_n=-\xi_{n-1}+(2+V''(u_n))\xi_n-\xi_{n+1}.$$
    Since $V'$ is an almost periodic function, it follows that $V''$ is bounded and $D^2\Phi(u)\in\mathscr L(\ell^2)$, where $\mathscr L(X)$ denotes the space of all bounded linear operators from $X$ to $X$.

    Next, we will prove that $D^2\Phi(u)$ is not invertible in $\mathscr L(\ell^2)$.

    Since for all $\beta\in\mathbb R$, $D\Phi(u(\beta))=0$, taking the derivative with respect to $\beta$, we obtain $$\frac{d}{d\beta}D\Phi(u(\beta))=D^2\Phi(u(\beta))((\frac{\partial u_n}{\partial\beta})_{n\in\mathbb Z})=0,$$ where $\frac{\partial u_n}{\partial \beta}=\partial_{\alpha}\hat h^*(n\omega\alpha+\beta\alpha) +1=\hat l^*(n\omega\alpha+\beta\alpha)$. $\hat l^*$ is defined in the KAM theorem, and it can be provided that the analytic norm of $\hat l^*$ is bounded. Therefore, $(\xi_n)_{n\in\mathbb Z}:=(\frac{\partial u_n}{\partial \beta})_{n\in\mathbb Z}\in \ell^{\infty}$ and $D^2\Phi(u)$ is not invertible in $\mathscr L(\ell^{\infty})$. 

   { Since $\|(\hat l^*(\sigma))^{-1}\|_{\frac{\rho_0}{2}}\le 2N^-$, we have $(\xi_n)_{n\in\mathbb Z}\ne 0$.} Since $D^2\Phi(u(\beta))((\xi_n)_{n\in\mathbb Z})=0$, if $(\xi_n)_{n\in\mathbb Z}\in l^2$, then $D^2\Phi(u)$ is not invertible in $\mathscr L(\ell^2)$. If $(\xi_n)_{n\in\mathbb Z}\in \ell^2$, that means $\|(\xi_n)_{n\in\mathbb Z}\|_{\ell^2}=+\infty$.

    Construct
$$ \xi_n^{[k]}:=\left\{
    \begin{aligned}
        \frac{\partial u_n}{\partial\beta}&,&|n|\le k,\\
        0\,\,\,\,&,&|n|>k.
    \end{aligned}
    \right.
$$
then $(\xi_n^{[k]})_{n\in\mathbb Z}\in l^2$. 
    The direct calculation yields
    $$
\eta_n^{[k]}:=(D^2\Phi(u)(\xi^{[k]}))_n=-\xi_{n-1}^{[k]}+(2+V''(u_n))\xi_n^{[k]}-\xi_{n+1}^{[k]}=\left\{
\begin{aligned}
    &-\xi_{n-1}+(2+V''(u_n))\xi_n,\quad n=k\\
    &(2+V''(u_n))\xi_n-\xi_{n+1},\quad\quad\, n=-k\\
    &-\xi_{n-1},\quad\quad\quad\quad\quad\quad\quad\quad\,\,\,\, n=k+1\\
    &-\xi_{n+1},\quad\quad\quad\quad\quad\quad\quad\quad\,\,\,\, n=-k-1\\
    &0,\quad\quad\quad\quad\quad\quad\quad\quad\,\,\,\,\quad\quad\quad\ else.
\end{aligned}
\right.
$$
 Note that there are only 4 terms in $\eta^{[k]}$ that are not zero. {Denote $A:=\sup_{\sigma\in\mathbb T^{\mathbb N}}|2+\partial_{\alpha}\partial_{\alpha}\widehat V(\sigma)|=\sup_{\theta\in\mathbb R}|2+V''(\theta)|$, $N^{*+}:=\sup_{\sigma\in\mathbb T^{\mathbb N}}|\hat l^*(\sigma)|$.} Then
 \begin{align*}
   \|\eta^{[k]}\|_{\ell^2}&=\sqrt{|-\xi_{k}|^2+|-\xi_{-k}|^2+|(2+V''(u_{-k}))\xi_{-k}-\xi_{-k+1}|^2+|-\xi_{k-1}+(2+V''(u_{k}))\xi_{k}|^2}\\
   &\le\sqrt{2(N^{*+})^2+2(AN^{*+}+N^{*+})^2}=:M,  
 \end{align*}
 where $M$ is independent of $k$.

 If $D^2\Phi(u)$ is invertible in $\mathscr L(\ell^2)$, then $\xi^{[k]}=(D^2\Phi(u))^{-1}\eta^{[k]}$.
 Moreover, we have
 \begin{equation*}
\begin{aligned}
    \|(D^2\Phi(u))^{-1}\|_{\mathscr L(\ell^2)}&=\sup_{\|\eta\|_{\ell^2}\le 1}\|(D^2\Phi(u))^{-1}\eta\|_{\ell^2}=\sup_{\|\eta\|_{\ell^2}\le M}\frac{\|(D^2\Phi(u))^{-1}\eta\|_{\ell^2}}{M}\\
    &\ge\frac{\|(D^2\Phi(u))^{-1}\eta^{[k]}\|_{\ell^2}}{M}=\frac{\|\xi^{[k]}\|_{\ell^2}}{M}.
\end{aligned}
\end{equation*}
Since $\|\xi^{[k]}\|_{\ell^2}\to\infty$ as $k\to\infty$, it follows that $\|(D^2\Phi(u))^{-1}\|_{\mathscr L(\ell^2)}$ is unbounded. Therefore, $D^2\Phi(u)$ is not invertible in the space of $\mathscr L(\ell^2)$. Consequently, the equilibrium configuration $u$ obtained by the KAM method does not have a phonon gap. 
\end{proof}

\section{Equilibria by anti-integrable limits with a phonon gap}
 In this section we aim to prove that the equilibrium configuration obtained using the anti-integrable limits has a  phonon gap.

 By anti-integrable limits, one can find the equilibrium configuration whose Lagrangian is given by 
 $$L(x,y)=I(x-y)+\lambda V(x), \quad x,y\in\mathbb R^d, \,I,V\in C^2(\mathbb R^d,\mathbb R).$$
 We begin by reviewing the anti-integrable limits theorem.

 \begin{defi}
    A continuous function $\psi:\mathbb R^d\to\mathbb R^d$ is said to satisfy Aubry criterion if there exists a subset $O\subset \mathbb R^d$ and constants $R>0$, $r>0$ and $m>0$, such that:
    \begin{itemize}
        \item [(1)] every ball with radius $R$ contains at least one point in $O$,
        \item [(2)] $\psi(z)=0$, $\forall z\in O$,
        \item [(3)] $\|\psi(x)-\psi(y)\|\ge m\|x-y\|$, $\forall z\in O$, $\forall x,y\in\overline{B}_r(z)$.
    \end{itemize}
 \end{defi}
 
 \begin{defi}
     Given a matrix $A$, we denote by $\sigma_{\max}(A)$ (resp. $\sigma_{\min}(A)$) the largest (resp. smallest) singular value of $A$.
 \end{defi}

 \begin{thm}[\cite{du2024anti}]\label{anti}
     Let $I,V\in C^2(\mathbb R^d,\mathbb R)$ and assume that $\nabla V$ satisfies the Aubry criterion with parameters $O,R,r,m$. If $\rho\in\rm{Hom}(\mathbb Z,\mathbb R^d)$, $\lambda\ge\frac{r+R}{rm}K(\rho,r+R)$, where $K(\rho,r+R):=4\sup_{x\in B_{\|\rho\|+2(R+r)}(0)}\sigma_{\max}(\nabla^2I(x))$, then there exists an equilibrium configuration $u$ satisfying 
     \begin{itemize}
         \item [(1)] $\sup_{n\in\mathbb Z}\inf_{z\in O}\|u_n-z\|\le r$ and $d(u,\rho)\le r+R$, 
         \item [(2)] $\sigma_{\max}(\nabla^2I(u_n-u_{n+1}))\le\frac{K(\rho,r+R)}{4}$, $\forall n\in\mathbb Z$,
         \item [(3)] $\sigma_{\min}(\nabla^2V(u_n))\ge m>0$, $\forall n\in\mathbb Z$. 
     \end{itemize}
     
     Moreover, the configuration $u$ is unique in the sense that if $u'$ is also an equilibrium configuration, and for all $i\in\mathbb Z$ there exists $x\in O$, such that $u_i,u_i'\in\overline{B}_r(x)$, then $u=u'$.
 \end{thm}

 \begin{rmk}
     Actually, \cite{du2024anti} presents a more general case, where the equation for the equilibrium state is given by $(\Delta+\lambda\Psi)u=0$, where $\Delta$ is an invariant operator and $\Psi$ is a pointwise operator.
 \end{rmk}


 \begin{thm}\label{anti-integrable limits}
     Assume that {$\sup_{d(x,\rho)\le r+R}\sup_{n\in\mathbb Z}\sigma_{\max}(\nabla^2V(x_n)) \leq M$ for some $M>0$}, then the equilibrium configuration obtained by Theorem \ref{anti} has a phonon gap. Moreover, the gap parameter $G(u)\ge\frac{\lambda m-K(\rho,r+R)}{\lambda M+K(\rho,r+R)}$.
 \end{thm}
 \begin{proof}
     In this setting,
     $\Phi(u)=\sum_nL(u_n,u_{n+1})=\sum_n I(u_n-u_{n+1})+\lambda V(u_n)$.
     Then
    \begin{equation*}
    \begin{aligned}
        &\quad \big(D^2\Phi(u)(\xi) \big)_n\\
        &=\big(-\nabla^2I(u_{n-1}-u_n)\big)\xi_{n-1}+\big(\nabla^2I(u_n-u_{n+1})+\lambda\nabla^2V(u_n)+\nabla^2I(u_{n-1}-u_n)\big)\xi_n\\
        &\quad+ \big(-\nabla^2I(u_n-u_{n+1}) \big)\xi_{n+1}\\
        &= \big(-\nabla^2I(u_{n-1}-u_n) \big)\xi_{n-1}+ \big(\nabla^2I(u_n-u_{n+1})+\nabla^2I(u_{n-1}-u_n) \big)\xi_n+ \big(-\nabla^2I(u_n-u_{n+1})\big)\xi_{n+1}\\
        &\quad+\lambda\nabla^2V(u_n)\xi_n.
    \end{aligned}
    \end{equation*}
    We need to prove that $D^2\Phi(u)$ is invertible in the space $\mathscr L(\ell^2)$.

    Define the linear operator 
    $A:\ell^2\to \ell^2$ by
    $$(A(\xi))_n=(-\nabla^2I(u_{n-1}-u_n))\xi_{n-1}+(\nabla^2I(u_n-u_{n+1})+\nabla^2I(u_{n-1}-u_n))\xi_n+(-\nabla^2I(u_n-u_{n+1}))\xi_{n+1}.$$
    Since $\sigma_{\max}(\nabla^2I(u_n-u_{n+1}))\le\sup_{x\in B_{\|\rho\|+2(R+r)}(0)}\sigma_{\max}(\nabla^2I(x))$, $A$ is well defined. Therefore, we have $\|A\|_{\mathscr L(\ell^2)}\le 4\max_{\|x\|\le2(r+R)+\|\rho\|}\sigma_{\max}(\nabla^2I(x))=K(\rho,r+R)$.

    Define the linear operator $B:\ell^2\to \ell^2$, $(B(\xi))_n=\nabla^2 V(u_n)\xi_n$. Since the equilibrium $u$ obtained by anti-integrable limit satisfies $M\ge\sigma_{\max}(\nabla^2V(u_n)\ge\sigma_{\min}(\nabla^2V(u_n))\ge m>0$, for all $n\in\mathbb Z$, it follows that $B$ is also well defined. $B$ is invertible in $\mathscr L(\ell^2)$, and $$\|B^{-1}\|_{\mathscr L(\ell^2)}=\sup_{\|\eta\|_{\ell^2}=1}\|B^{-1}\eta\|_{\ell^2}=\sup_{\|\xi\|_{\ell^2}=1}\frac{\|B^{-1}B\xi\|_{\ell^2}}{\|B\xi\|_{\ell^2}}=\frac{1}{\inf_{\|\xi\|_{\ell^2}=1}\|B\xi\|_{\ell^2}}\le\frac1m,$$
    where $\|B\xi\|_{\ell^2}=(\sum_n\|\nabla^2V(u_n)\xi_n\|^2)^{\frac12}\ge(\sum_nm^2\|\xi_n\|^2)^{\frac12}=m\|\xi\|_{\ell^2}$.

    Therefore, $D^2\Phi(u)=A+\lambda B=\lambda B(I+\lambda^{-1}B^{-1}A)$, where $\|\lambda^{-1}B^{-1}A\|_{\mathscr L(\ell^2)}\le\frac1{m\lambda}K(\rho,r+R)$, and $D^2\Phi(u)$ is invertible, if $\lambda>\frac1m K(\rho,r+R)$. Consequently, the equilibrium configuration $u$ obtained by anti-integrable limit possesses a phonon gap. 

    Moreover, we have
    \begin{align*}
        \sigma_{\max}(D^2\Phi(u))\le\|D^2\Phi(u)\|_{\mathscr L(\ell^2)}\le\|A\|_{\mathscr L(\ell^2)}+\lambda\|B\|_{\mathscr L(\ell^2)}\le K(\rho,r+R)+\lambda M,
    \end{align*}
    and
    \begin{align*}
        \sigma_{\max}(D^2\Phi(u)^{-1})&\le\|D^2\Phi(u)^{-1}\|_{\mathscr L(\ell^2)}\le\lambda^{-1}\|B^{-1}\|_{\mathscr L(\ell^2)}\|(I+\lambda^{-1}B^{-1}A)^{-1}\|_{\mathscr L(\ell^2)}\\
        &\le\lambda^{-1}\frac1m \frac{1}{1-\frac{1}{m\lambda}K(\rho,r+R)}=\frac{1}{\lambda m-K(\rho,r+R)}.
    \end{align*}
    Then the gap parameter $G(u)=\frac{1}{\sigma_{\max}(D^2\Phi(u)^{-1})\sigma_{\max}(D^2\Phi(u))}\ge\frac{\lambda m-K(\rho,r+R)}{\lambda M+K(\rho,r+R)}$.
 \end{proof}

\bibliographystyle{alpha}
\bibliography{reference}

\end{document}